\documentclass[12pt]{article}

\usepackage{amsmath,amssymb,amscd,epic}
\usepackage{bbm,amsthm}
\usepackage{mathtools}
\usepackage{leftidx}
\usepackage{hyperref}
\usepackage{mathrsfs}
\usepackage{here}
\usepackage{lscape}
\usepackage{color}
\usepackage[top=30truemm,bottom=30truemm,left=25truemm,right=25truemm]{geometry}

\usepackage{caption}
\usepackage{subcaption}

\usepackage{graphicx} 

\theoremstyle{plain}
\newtheorem{thm}{Theorem}[section]

\newtheorem{prop}[thm]{Proposition}

\theoremstyle{definition}

\newtheorem{question}{Question}[section]
\newtheorem{rem}{Remark}[section]

\title{Wolstenholme primes and group determinants \\ of cyclic groups} 
\author{Cid Reyes-Bustos, Naoya Yamaguchi and Yuka Yamaguchi} 
\date{\today}

\begin{document}

\maketitle

\begin{abstract}
  A Wolstenholme prime is a prime number $p \geq 5$ that divides the numerator of the Bernoulli number $B_{p-3}$. A number of
  equivalent definitions for Wolstenholme primes are known, mostly related to congruences of harmonic sums or binomial coefficients.
  In this paper, we introduce an equivalent definition of Wolstelholme primes related the number of terms in the group determinant of cyclic groups, and equivalently, the cardinality of certain sets of restricted partitions.
    
  
\end{abstract}

\section{Introduction}

In 1862, Wolstenholme proved that for a prime number $p\geq5$, the numerator of the harmonic sum
\begin{align}
  \label{eq:wolsthmH}
  H_{p-1} :=  1 + \frac12 + \frac13 + \cdots + \frac1{p-1}
\end{align}
is divisible by $p^2$. This theorem, known as Wolstenholme theorem, is equivalent to the congruence
\begin{align}
  \label{eq:wolsthm}
  \binom{2p - 1 }{p-1} \equiv 1 \pmod{p^3}.
\end{align}
There is a number of related results and extensions of Wolstenholme theorem, 
we refer the reader to \cite{mestrovic2011wolstenholmes} for an extensive review up to 2012.

It is immediate to verify that the congruence \eqref{eq:wolsthm} in Wolstenholme theorem does not hold in general for
higher powers of $p$. 
For instance, we have
\[
  \binom{2p - 1 }{p-1} \equiv 126 \pmod{p^4}
\]
for $p=5$.  
In fact, a prime number $p$ satisfying the congruence relation
\begin{align}
  \label{eq:wolsprime}
\binom{2p-1}{p-1} \equiv 1 \pmod{p^{4}} 
\end{align}
is called Wolstenholme prime \cite[p.385]{MR1339137}. The only known Wolstenholme primes are $16843$ and $2124679$ \cite[pp.915--916]{MR4437433}, but it is conjectured that the set of Wolstenholme primes is infinite based on arguments using the prime number theorem \cite[p.387]{MR1339137}. We remark that it is also conjectured that no prime number $p$ satisfies a congruence of type \eqref{eq:wolsprime} for $p^5$.

Similar to the case of Wolstenholme theorem, there are a number of equivalent definitions for Wolstenholme primes \cite{mestrovic2011wolstenholmes}. For instance, a  prime number $p$ is a Wolstenholme prime if it satisfies any of 
the following conditions:
\begin{enumerate}
\item[(a)] the numerator of the harmonic sum $H_{p-1}$ is divisible by $p^{3}$ (cf. \eqref{eq:wolsthmH}), 
\item[(b)] the prime $p$ divides the numerator of the Bernoulli number $B_{p-3}$,
\item[(c)] $p$ and $p - 3$ are irregular primes (in this case, $(p,p-3)$ is called an irregular pair).
\end{enumerate}

In this paper, we give a sufficient and necessary condition for a prime $p$ to be Wolstenholme prime related
to number of terms of the group determinant of cyclic groups, and the number of certain restricted partitions. 

Let us recall the definition of group determinant. For a finite group $G$ and a set of variables
$\{x_g\}_{g \in G}$ indexed by the group elements, the determinant
\[
  \Theta(G) := \det{\left( x_{g h^{- 1}} \right)}_{g, h \in G}
\]
is called the group determinant of $G$. The group determinant $\Theta(G)$ actually determines the group itself \cite{MR1062831,MR1123661,MR4227663}, and contains considerable information about the group structure. Notably, it is known that the factorization of $\Theta(G)$ into linear factors is equivalent to the decomposition of the regular representation of the group into irreducible representations. The problem of factorization of the group determinant was originally considered in the XIX century for abelian groups by Dedekind, and by Frobenius in the general case. We refer the reader to \cite{Hawkins} for historical remarks and to \cite{Johnson} for a modern account of the theory of group determinant.

The main result of this paper is that $p$ is a Wolstenholme prime if and only if the congruence
\[
  {\rm N}(\Theta({\rm C}_{p})) \equiv 1 \pmod{p^{3}} 
\]
holds. Here, ${\rm C}_{n}$ is the cyclic group of order $n$ and $N(f)$ denotes the number of terms in a polynomial $f$. We also give a formulation of Wolstenholme theorem in terms of the group determinant of ${\rm C}_{n}$. 

This result is obtained by considering the explicit expression of the number of terms of the group determinant, which is related to a certain type of restricted partitions, which are generalizations of restricted partitions used in combinatorics, for instance, for the Frobenius problem, also known as coin-exchange problem \cite{BeckRobins}. The results of this paper may be interpreted as a group theoretical, and combinatorial, interpretations of the definition of Wolstenholme primes.

\section{Main results}

In order to state the main results, we start by recalling certain results on the number of terms of group determinant and its relation with restricted partitions. 

The group determinant $\Theta({\rm C}_{n})$ of ${\rm C}_{n}$ is given by
\[
    \Theta({\rm C}_{n}) = \prod_{i=0}^{n-1} \left( \sum_{j=0}^{n-1} \zeta_n^{i j} x_j \right),
\]
where $\zeta_n$ is a primitive $n$-th root of unity. This result is equivalent to the fact that the characters of ${\rm C}_{n}$ are given by $\chi_i(j)= \zeta^{i j}_n$ for $j \in {\rm C}_{n}$ and $0 \leq i \leq n - 1$. 
It is also not difficult to verify this result from the properties of circulant matrices.

Next, define the set $\tilde{\Lambda}_{n}^{k}$ to be the set of restricted partitions by
\[
\tilde{\Lambda}_{n}^{k} := \left\{ (\lambda_{1}, \lambda_{2}, \ldots, \lambda_{k n}) \in \mathbb{Z}^{kn} \mid 1 \leq \lambda_{1} \leq \lambda_{2} \leq \cdots \leq \lambda_{kn} \leq n, \quad 
\sum_{i = 1}^{k n} \lambda_{i} \equiv 0 \pmod{n} \right\}.
\]
The set $\tilde{\Lambda}_{n}^{k}$ is finite and the cardinality of $\tilde{\Lambda}_{n}^{k}$ is seen to be
\begin{equation}
  \label{eq:cardLambda}
  \frac{1}{n} \sum_{d \mid n} \binom{d k + d - 1}{d - 1} \varphi \left( \frac{n}{d} \right)
\end{equation}
using the formulas given in \cite{MR1691428} for dimensions of vector spaces of homogeneous invariants of cyclic groups.
Here $\varphi$ denotes the Euler totient function.

The relation of the restricted partitions 
with the number of terms of the group determinant of ${\rm C}_{n}$ is given by the inequality
\begin{equation}
  \label{eq:NGDI}
  {\rm N}(\Theta({\rm C}_{n})^{k}) \leq | \tilde{\Lambda}_{n}^{k} |,
\end{equation}
for $k\geq 1$. 
We refer the reader to \cite{https://doi.org/10.48550/arxiv.2203.14422} for congruences related to $|\tilde{\Lambda}_{n}^{k}|$ and 
${\rm N}(\Theta({\rm C}_{n})^{k})$, and in particular to Remark~3.3 for the proof of the inequality \eqref{eq:NGDI}.

In particular, we note (\cite[Corollary~4.2]{https://doi.org/10.48550/arxiv.2203.14422}) that equality in \eqref{eq:NGDI} is achieved when $n$ is a prime $p$ and $k = 1$, that is,
\begin{align}
  \label{eq:NGDL}
 {\rm N}(\Theta({\rm C}_{p})) = | \tilde{\Lambda}_{p}^{1} |. 
\end{align}

With these preparations, we can state the results of this paper. First, we give the group determinant version of Wolstenholme theorem.

\begin{thm}\label{cor:4}
A prime number $p \geq 5$ satisfies
\begin{align*}
  {\rm N}(\Theta({\rm C}_{p})) \equiv 1 \pmod{p^{2}}.
\end{align*}
\end{thm}

Next, we give a sufficient and necessary condition for a prime number to be a Wolstenholme prime. Note that in both
results, the exponent in the congruence is smaller than the corresponding results (i.e. \eqref{eq:wolsthm} and \eqref{eq:wolsprime}, respectively).

\begin{thm}\label{thm:1}
A prime number $p$ is a Wolstenholme prime if and only if 
\[
  {\rm N}(\Theta({\rm C}_{p})) \equiv 1 \pmod{p^{3}}
\]
holds. 
\end{thm}

The proofs of the theorems are based on the corresponding version of Wolstenholme theorem for restricted partitions.

\begin{prop} \label{prop:3}
  For prime $p \geq 5$ and $k, l \geq 1$, the congruence
  \[
    | \tilde{\Lambda}_{p^{l}}^{k} | \equiv 1 \pmod{p^{2}}
  \] holds.
  In addition, if $p$ is a Wolstenholme prime, then we have 
  \[
    | \tilde{\Lambda}_{p^{l}}^{1} | \equiv 1 \pmod{p^{3}}.
  \]  
\end{prop}

\begin{proof}
Let $k' := k + 1$. From~\cite[Remark~3.3]{https://doi.org/10.48550/arxiv.2203.14422} we have
\begin{align*}
p^{l} | \tilde{\Lambda}_{p^{l}}^{k} | 
&= \sum_{d \mid p^{l}} \binom{k' d - 1}{d - 1} \varphi(\frac{p^{l}}{d}) \\ 
&= \sum_{i = 0}^{l} \binom{k' p^{i} - 1}{p^{i} - 1} \varphi(p^{l - i}) \\ 
&= p^{l} - p^{l - 1} + \sum_{i = 1}^{l - 1} \binom{k' p^{i} - 1}{p^{i} - 1} (p^{l - i} - p^{l - i - 1}) + \binom{k' p^{l} - 1}{p^{l} - 1} \\ 
&= p^{l} - p^{l - 1} + \sum_{i = 1}^{l - 1} \binom{k' p^{i} - 1}{p^{i} - 1} p^{l - i} - \sum_{i = 2}^{l} \binom{k' p^{i - 1} - 1}{p^{i - 1} - 1} p^{l - i} + \binom{k' p^{l} - 1}{p^{l} - 1} \\ 
&= p^{l} + \sum_{i = 1}^{l} \left\{ \binom{k' p^{i} - 1}{p^{i} - 1} - \binom{k' p^{i - 1} - 1}{p^{i - 1} - 1} \right\} p^{l - i}. 
\end{align*}
Here, from \cite[Eq~(39)]{mestrovic2011wolstenholmes}, there is an  $a_{i} \in \mathbb{Z}$ such that
$$
\binom{k' p^{i} - 1}{p^{i} - 1} - \binom{k' p^{i - 1} - 1}{p^{i - 1} - 1} = p^{3 i} a_{i}
$$
then
$$
| \tilde{\Lambda}_{p^{l}}^{k} | = \frac{1}{p^{l}} \left( p^{l} + \sum_{i = 1}^{l} p^{l - i} p^{3 i} a_{i} \right) = 1 + \sum_{i = 1}^{l} p^{2 i} a_{i}
$$
holds. 
In addition, from
\begin{align*}
p^{l} | \tilde{\Lambda}_{p^{l}}^{1} |  = p^{l} + \left( \binom{2 p - 1}{p - 1} - 1 \right) p^{l - 1} + \sum_{i = 2}^{l} p^{l + 2 i} a_{i}, 
\end{align*}
if $p$ is a Wolstenholme prime, then $| \tilde{\Lambda}_{p^{l}}^{1} | \equiv 1 \pmod{p^{3}}$ holds.
\end{proof}

Note that by the equality \eqref{eq:NGDL}, this proves Theorem~\ref{cor:4} and one of the implications in
Theorem~\ref{thm:1}. 

\begin{proof}[Proof of Theorem~$\ref{thm:1}$]
From Proposition~\ref{prop:3}, it is enough to prove that if the prime $p$ satisfies ${\rm N}(\Theta({\rm C}_{p})) \equiv 1 \pmod{p^{3}}$, then 
$p$ is a Wolstenholme prime.  For $p = 2, 3$, we have ${\rm N}(\Theta({\rm C}_{p})) = 2, 4$, therefore we assume $p \geq 5$. We have
\begin{align*}
1 \equiv {\rm N}(\Theta({\rm C}_{p})) \equiv \frac{1}{p} \left( p - 1 + \binom{2 p - 1}{p - 1} \right) \pmod{p^{3}}
\end{align*}
and we conclude that $p$ is Wolstenholme prime. 
\end{proof}

\begin{rem}
Theorem~$\ref{cor:4}$ is a generalization of the congruence
\begin{align*}
{\rm N}(\Theta({\rm C}_{p})) \equiv 1 \pmod{p} 
\end{align*}
obtained in ~\cite{MR4593070}.
\end{rem}

\section{Possible generalizations and numerical computations}

In view of Proposition \ref{prop:3} and inequality \eqref{eq:NGDI} it is natural to ask if for a prime number $p\geq 5$ the ``Wolstenholme type'' congruence relation
\begin{align}
  \label{eq:gdcong2}
  {\rm N}(\Theta({\rm C}_{p^k})^{l}) \equiv 1 \pmod{p^{2}}
\end{align}
holds for $k, l \geq 1$. 
Regarding this situation, we first reproduce the questions posed in \cite{MR4593070}.

\begin{question}[{\cite[Question~3.1, 3.2]{MR4593070}}]
\label{question}
Are the following statements true? 
\begin{enumerate}
\item[1.] For any positive integer $k$, the equality ${\rm N} \left( \Theta({\rm C}_{n})^{k} \right) = | \tilde{\Lambda}_{n}^{k} |$ holds if and only if $n$ is a prime power. 
\item[2.] For any positive integers $n$ and $k$, the equality ${\rm N} \left( \Theta({\rm C}_{n})^{k} \right) \equiv | \tilde{\Lambda}_{n}^{k} | \pmod{n}$ holds. 
\item[3.] For any group $G$ of order $n$, 
${\rm N} \left( \Theta({\rm C}_{n}) \right) \leq {\rm N} \left( \Theta(G) \right)$ holds. 
\item[4.] If ${\rm N} \left( \Theta(G) \right) = {\rm N} \left( \Theta(G') \right)$ holds for arbitrary groups $G$ and $G'$ of the same order, 
then $G$ is isomorphic to $G'$ as group. 
\item[5.] For any abelian group $G$ and any non abelian group $G'$ of the same order, 
\[ 
{\rm N} \left( \Theta(G) \right) < {\rm N} \left( \Theta(G') \right)
\]
holds. 
\end{enumerate}
\end{question}

By Proposition \ref{prop:3}, Question \ref{question}.1 clarifies the congruence \eqref{eq:gdcong2} for $p\geq 5$. To give some insight for the validity of the posed questions, we computed  ${\rm N}(\Theta({\rm C}_{n})^{k})$ for cyclic groups with different values of $n,k$. 

The results are shown in Tables \ref{tab:terms2}, \ref{tab:groupdet1} and \ref{tab:partition1}, where we also give the cardinality of $\tilde{\Lambda}_{n}^{k}$ for reference. As we can observe from the results, Questions \ref{question}.1--2 holds in all the computed cases. 
It would also be interesting to explore the congruence properties of the number of terms of the group determinant of other groups (e.g. abelian groups). 

\begin{table}[h]
    \begin{subtable}[h]{0.45\textwidth}
        \centering
        \begin{tabular}{c | l  l  l}
        $n \symbol{92} k$& 1 & 2  \\
        \hline \hline
         10 & 7492 & 996483   \\ 
         11 & 32066 & 5864750    \\
         12 & 86500 & 34724470  \\
      13 & 400024 & 208267320    \\
      14 & 1366500 & 1258462082  
       \end{tabular}
       \caption{$\Theta({\rm C}_{n})^{k}$.  }
       \label{tab:week1}
    \end{subtable}
    \hfill
    \begin{subtable}[h]{0.45\textwidth}
        \centering
        \begin{tabular}{c | l l l}
        $n \symbol{92} k$& 1 & 2  \\
        \hline \hline
         10 & 9252 & 1001603   \\ 
         11 & 32066 & 5864750   \\
         12 & 112720 & 34769374   \\
      13 & 400024 & 208267320   \\
      14 & 1432860 & 1258579654 
        \end{tabular}
        \caption{$\tilde{\Lambda}_{n}^{k}$.  }
        \label{tab:week2}
     \end{subtable}
     \caption{Number of terms and cardinal number.  }
     \label{tab:terms2}
\end{table}

To verify Questions \ref{question}.3--5, we computed ${\rm N}(\Theta(G))$ for all the groups of order $16$. The number of terms of the group determinant are shown in Table \ref{tab:groupdet3} where we also include the group code in the GAP computer algebra system \cite{GAP4}.  In this case, we also verify the positive answer for all posed questions. For groups of smaller order we refer the reader to \cite{MR4593070}.

\begin{table}[ht]
    \centering
    \begin{tabular}{c | c|l }
         GAP & $G$ & ${\rm N}(\Theta(G))$  \\
      \hline
      \hline
       1  & ${\rm C}_{16}$ & 18784170 \\ 
       5  & ${\rm C}_{8} \times {\rm C}_{2}$ & 18784979 \\ 
       2  & ${\rm C}_{4} \times {\rm C}_{4}$ & 18784995 \\ 
       10  & ${\rm C}_{4} \times {\rm C}_{2}^{2}$ & 18786595 \\ 
       14  & ${\rm C}_{2}^{4}$ & 18789795 \\ 
       11  & ${\rm D}_{8} \times {\rm C}_{2}$ & 36768531 \\ 
       13 & ${\rm Q}_{8} \rtimes {\rm C}_{2}$   & 36808747 \\ 
       3  & ${\rm C}_{2}^{2} \rtimes {\rm C}_{4}$ & 36811299 \\ 
       4  & ${\rm C}_{4} \rtimes {\rm C}_{4}$ & 36819043 \\ 
       6  & ${\rm C}_{8} \rtimes_{5} {\rm C}_{2}$  & 36842795 \\ 
       12  & ${\rm Q}_{8} \times {\rm C}_{2}$ & 36855987 \\ 
       8  & ${\rm C}_{8} \rtimes_{3} {\rm C}_{2}$ & 73395796 \\ 
       9  & ${\rm Q}_{16}$ & 73432499 \\ 
       7  & ${\rm D}_{16}$ & 73455914 
    \end{tabular}
    \caption{Number of terms in $\Theta(G)$ for groups of order $16$.  }
    \label{tab:groupdet3}
\end{table}

In Table \ref{tab:groupdet3}, ${\rm D}_n$ is the dihedral group and ${\rm Q}_n$ is the (generalized) quaternionic group of order $n$, respectively. In addition, we have
\begin{align*}
{\rm C}_{2}^{2} \rtimes {\rm C}_{4} &:= \langle g_{1}, g_{2}, g_{3} \mid g_{1}^{2} = g_{2}^{2} = g_{3}^{4} = e, \: g_{2} g_{1} = g_{1} g_{2}, \: g_{3} g_{1} = g_{1} g_{3}, \: g_{3} g_{2} = g_{1} g_{2} g_{3}  \rangle, \\
{\rm C}_{4} \rtimes {\rm C}_{4} &:= \langle g_{1}, g_{2} \mid g_{1}^{4} = g_{2}^{4} = e, \: g_{2} g_{1} = g_{1}^{3} g_{2} \rangle, \\
{\rm C}_{8} \rtimes_{5} {\rm C}_{2} &:= \langle g_{1}, g_{2} \mid g_{1}^{8} = g_{2}^{2} = e, \: g_{2} g_{1} = g_{1}^{5} g_{2} \rangle, \\ 
{\rm C}_{8} \rtimes_{3} {\rm C}_{2} 
&:= \langle g_{1}, g_{2} \mid g_{1}^{8} = g_{2}^{2} = e, \: g_{2} g_{1} = g_{1}^{3} g_{2} \rangle, \\ 
{\rm Q}_{8} \rtimes {\rm C}_{2} 
&:= \langle g_{1}, g_{2}, g_{3} \mid g_{1}^{4} = g_{3}^{2} = e, \: g_{1}^{2} = g_{2}^{2}, \: g_{2} g_{1} = g_{1} g_{2}, \: g_{3} g_{2} = g_{2} g_{3}, \: g_{3} g_{1} = g_{1}^{3} g_{3} \rangle. 
\end{align*}
Here, $e$ denotes the unit element of each group. We refer to \cite{Wild2005} for the extended discussion on the structure of the groups of order $16$.


\section*{Acknowledgements}

The computations were done in an Intel Xeon E7-4809 2Ghz with 1TB memory using SAGE-Math. CRB would like to thank Hideki Sakurada (NTT CS Labs) for the technical support for the computations.

\begin{landscape}
\begin{table}[ht]
    \centering
    \begin{tabular}{c|lllllllllll}
         $n \symbol{92} k$& 1 & 2 & 3 & 4 & 5 & 6 & 7 & 8 & 9 & 10 \\
      \hline
      \hline
         1 & 1  &  1 & 1  & 1  & 1  & 1  & 1  & 1  & 1 & 1  \\
         2 & 2 & 3 & 4 & 5 & 6 & 7 & 8 & 9 & 10 & 11 \\ 
         3 & 4 & 10 & 19 & 31 & 46 & 64 & 85 & 109 & 136 & 166 \\
         4 & 10 & 43 & 116 & 245 & 446 & 735 & 1128 & 1641 & 2290 & 3091 \\
         5 & 26 & 201 & 776 & 2126 & 4751 & 9276 & 16451 & 27151 & 42376 & 63251 \\
         6 &  68 & 984 & 5566 & 19751 & 53994 & 124900 & 255614 & 478305 & 834454 & 1376666 \\
         7 & 246 & 5538 & 42288 & 192130 & 642342 & 1753074 & 4141383 & 8782075 & 17125354 & 31231278 \\
         8 & 810 & 30667 & 328756 & 1922741 & 7861662 & 25366335 & 69159400 & 166237161 & 362345362 & 730421043 \\
         9 & 2704 & 173593 &  2615104 & 19692535 & 98480332 & 375677659 & 1182125128 & * & * & *
    \end{tabular}
    \caption{Number of terms in  $\Theta({\rm C}_{n})^{k}$. Here, the symbol $*$ indicates that the number of terms was not computed.}
    \label{tab:groupdet1} 
\vspace{1cm}
    \centering
    \begin{tabular}{c|lllllllllll}
         $n \symbol{92} k$& 1 & 2 & 3 & 4 & 5 & 6 & 7 & 8 & 9 & 10 \\
      \hline
      \hline
         1 & 1  &  1 & 1  & 1  & 1  & 1  & 1  & 1  & 1 & 1  \\
         2 & 2 & 3 & 4 & 5 & 6 & 7 & 8 & 9 & 10 & 11 \\ 
         3 & 4 & 10 & 19 & 31 & 46 & 64 & 85 & 109 & 136 & 166 \\
         4 & 10 & 43 & 116 & 245 & 446 & 735 & 1128 & 1641 & 2290 & 3091 \\
         5 & 26 & 201 & 776 & 2126 & 4751 & 9276 & 16451 & 27151 & 42376 & 63251 \\
         6 &  80 & 1038 & 5620 & 19811 & 54132 & 124936 & 255704 & 478341 & 834472 & 1376738 \\
         7 & 246 & 5538 & 42288 & 192130 & 642342 & 1753074 & 4141383 & 8782075 & 17125354 & 31231278 \\
         8 & 810 & 30667 & 328756 & 1922741 & 7861662 & 25366335 & 69159400 & 166237161 & 362345362 & 730421043 \\ 
         9 & 2704 & 173593 & 2615104 & 19692535 & 98480332&375677659 & 1182125128&3220837534 & 7847250409 & 17485161178 \\ 
    \end{tabular}
    \caption{Cardinal number of $\tilde{\Lambda}_{n}^{k}$.  }
    \label{tab:partition1}
\end{table}
\end{landscape}

\bibliography{reference}
\bibliographystyle{plain}

\smallskip

\flushleft

Cid Reyes-Bustos \par
\textsc{NTT Institute for Fundamental Mathematics, \\NTT Communication Science Laboratories, Nippon Telegraph and Telephone Corporation, \\3-9-11 Midori-cho Musashino-shi, Tokyo, 180-8585, Japan}

\texttt{cid.reyes@ntt.com}

\bigskip

Naoya Yamaguchi\par
\textsc{Faculty of Education, \\ 
University of Miyazaki, \\ 
1-1 Gakuen Kibanadai-nishi, Miyazaki, 889-2192, Japan
}

\texttt{n-yamaguchi@cc.miyazaki-u.ac.jp}

\bigskip

Yuka Yamaguchi\par
\textsc{Faculty of Education, \\ 
University of Miyazaki, \\ 
1-1 Gakuen Kibanadai-nishi, Miyazaki, 889-2192, Japan
}

\texttt{y-yamaguchi@cc.miyazaki-u.ac.jp}

\end{document}